\documentclass{amsart}
\usepackage{mathpazo}
\usepackage{amssymb}

\newtheorem{theorem}{Theorem}[section]
\newtheorem{lemma}[theorem]{Lemma}
\newtheorem*{Acknowledgement}{\textnormal{\textbf{Acknowledgement}}}
\theoremstyle{definition}
\newtheorem{definition}[theorem]{Definition}
\newtheorem{example}[theorem]{Example}

\newtheorem{corollary}[theorem]{Corollary}
\newtheorem{proposition}[theorem]{Proposition}
\newtheorem{remark}[theorem]{Remark}

\numberwithin{equation}{section}
\newcommand{\beqa}{\begin{eqnarray*}}
\newcommand{\eeqa}{\end{eqnarray*}}
\newcommand{\beqn}{\begin{eqnarray}}
\newcommand{\eeqn}{\end{eqnarray}}

\newcommand{\ov}{\overline}

\newcommand{\ci}{\subseteq}

\renewcommand{\a}{\alpha}

\newcommand{\e}{\varepsilon}

\newcommand{\la}{\lambda}

\newcounter{cnt1}
\newcounter{cnt2}
\newcounter{cnt3}
\newcommand{\blr}{\begin{list}{$($\roman{cnt1}$)$}
        {\usecounter{cnt1} \setlength{\topsep}{0pt}
                \setlength{\itemsep}{0pt}}}
\newcommand{\bla}{\begin{list}{$($\alph{cnt2}$)$}
        {\usecounter{cnt2} \setlength{\topsep}{0pt}
                \setlength{\itemsep}{0pt}}}
\newcommand{\bln}{\begin{list}{$($\arabic{cnt3}$)$}
        {\usecounter{cnt3} \setlength{\topsep}{0pt}
                \setlength{\itemsep}{0pt}}}
\newcommand{\el}{\end{list}}
\newtheorem{thm}{Theorem}

\newtheorem{Def}[thm]{Definition}

\newtheorem{rem}[thm]{Remark}
\newcommand{\Rem}{\begin{rem} \rm}
\newcommand{\bdfn}{\begin{Def} \rm}
\newcommand{\edfn}{\end{Def}}

\title{Small Combination of Slices, Dentability and Stability Results Of  Small Diameter  Properties In Banach Spaces}
\author[ S. Basu , S. Seal ]
	{Sudeshna Basu$^{1}$, Susmita Seal$ ^{2}$ }
\address{{$^{1}$}   Sudeshna Basu,
	Department of Mathematics,
	George Washington University,
	Washington DC 20052 USA  and \\
		Department of Mathematics, 
		Ram Krishna Mission Vivekananda Education and Research Institute , 
		Belur Math,  Howrah 711202
		West Bengal, India }
	\email{sudeshnamelody@gmail.com}

\address {{$^{2}$} Susmita Seal, 
		Department of Mathematics, ,
		Ram Krishna Mission Vivekananda Education and Research Institute , 
		Belur Math,  Howrah 711202,
		West Bengal, India}
	\email{susmitaseal1996@gmail.com}

\date{}
\begin{document}
\maketitle
\begin{abstract}
	In this work we study three different versions of small diameter properties of the unit ball in a Banach space and its dual. The related concepts for all closed bounded convex sets of a Banach space was initiated and developed in \cite{B3}, \cite{BR} ,\cite{EW}, \cite{GM} was extensively studied  in the context of dentability, huskability, Radon Nikodym Property and Krein Milman Property  in \cite{GGMS}. We introduce the the Ball Huskable Property ($BHP$), namely, the unit ball has  relatively weakly open subsets of arbitrarily small diameter. We compare this property to two related properties, $BSCSP$ namely, the unit ball has  convex combination of slices of arbitrarily small diameter and $BDP$ namely, the closed unit ball has slices of arbitrarily small diameter.  We  show  $BDP$ implies  $BHP$ which in turn implies  $BSCSP$ and none of the implications can be reversed. We prove similar results for  the $w^*$-versions. We prove that all these properties are stable under $l_p$ sum for $1\leq p \leq \infty, c_0$ sum and Lebesgue Bochner spaces. Finally, we explore the stability of these with properties in the light of three space property. We  show that $BHP$ is a three space property provided $X/Y$ is finite dimensional and same is true for $BSCSP$ when $X$ has $BSCSP$ and $X/Y$ is strongly regular (\cite{GGMS}).
\end{abstract}

\section{Introduction}

Let $X$ be a {\it real} 
nontrivial  Banach space and $X^*$ its dual. We will denote by $B_X$, $S_X$ and $B_X(x, r)$ the closed unit ball, the unit sphere and the closed ball of radius $r >0$ and center $x$. We refer to the monograph \cite{B1} for notions of convexity theory that we will be using here.

\bdfn
\blr
\item We say $A \ci B_{X^*}$ is a norming set for $X$ if $\|x\| =
\sup\{|x^*(x)| : x^* \in A\}$, for all $x \in X$. A closed subspace $F
\ci X^*$ is a norming subspace if $B_F$ is a norming set for $X$.
\item Let $x^* \in X^*$, $\a > 0$ and bounded set $ C \subseteq X$.
Then the set $S(C, x^*, \a) = \{x \in C : x^*(x) > \mbox{sup}~ x^*(C) - \a \}$ is called the  
slice determined by $x^*$ and $\a.$   
For any slice we assume without loss of generality that $\|x^*\| = 1$ , because 
 if $x^*\neq 0$ then $S(C,x^*,\alpha)=S(C,\frac{x^*}{\|x^*\|} ,\frac{\alpha}{\|x^*\|}).$ 
One can analogously define $w^*$ slices in $X^*$ by choosing the functional from the predual.

\item
A point $x$
 in a bounded set $C \ci X$ is called a denting point 
point of $C,$ if for every $\e > 0$, there exists a slice
$S$ of $C$, such that $x \in S$ and $dia(S) <\e.$ One can analogously define $w^*$-denting point of a bounded set in $X^*$ by considering $w^*$-slices.

\item A point $x$
in a bounded convex set $C \ci X$ is called a small combination of slices (SCS)
point of $C$, if for every $\e > 0$, there exists a convex combination of slices ,
$S = \sum_{i=1}^{n} \la_i S_i$ of $C$ (where $ 0 \leq \lambda_i \leq 1 $ and $\sum_{i=1}^{n} \la_i = 1$) such that $x \in S$ and $dia(S) <\e.$ One can analogously define $w^*$-SCS point of a bounded set in $X^*$ by considering $w^*$-slices.
\el
\edfn

 We recall the following two definitions from \cite{BR} and \cite{B2} .

\begin{definition} \label{def bdp} A Banach space $X$ has 
	
	\begin{enumerate}
		\item { \it Ball Dentable Property} ($BDP$) if $B_X$
		has a slice of arbitrarily small diameter. 
		\item {\it Ball Small Combination of Slice Property} ($BSCSP$) if $B_X$
		has a convex combination of slices of arbitrarily small diameter.
	\end{enumerate} 
	
\end{definition}
We now define,

\begin{definition}
	A Banach space $X$ has {\it  Ball Huskable Property}($BHP$) if $B_X$ 
	has a relatively weakly open subset of arbitrarily small diameter.
\end{definition}
\begin{remark}
Analogously we can define $w^{*}$-$BSCSP$, $w^{*}$-$BHP$ and $w^*$-$BDP$ in a dual space by considering $w^*$-SCS, $w^*$ open sets and  $w^*$- slices of $B_{X^*}$ respectively.
\end{remark}
Observe that for a Banach space , $BDP$ always implies $BHP$, in fact, any slice of the unit ball is relatively weakly open. Also $BHP$ implies $BSCSP$,
by Bourgain's Lemma (see\cite{B3}, \cite{GGMS}), which says that every non-empty relatively weakly open subset of $B_X$ contains a finite convex combination of slices . Similar observations are true for $w^*$-versions. Since every $w^*$-slice ($w^*$-open set) of $B_{X^*}$ is also a slice (weakly open set) of $B_{X^*}$, so we have the following diagram :
$$ BDP \Longrightarrow \quad BHP \Longrightarrow \quad  BSCSP$$ $\quad \quad \quad \quad \quad \quad \quad \quad\quad\quad\quad\quad\quad\quad\quad\quad \Big \Uparrow \quad \quad\quad\quad\quad \Big \Uparrow \quad \quad\quad\quad\quad \Big \Uparrow$  $$ w^*BDP \Longrightarrow  w^*BHP \Longrightarrow  w^*BSCSP$$
In general, none of the reverse implications of the diagram hold, which we will discuss later.
 
 SCS points were first introduced in \cite{GGMS} as a “slice generalization” of denting points as well as  the
 notion PC (i.e. points for which the identity mapping on the unit ball, from
 weak topology to norm topology is continuous)and subsequently analyzed in detail in \cite{R} and \cite{S}. It is well known that  $X$ has  Radon Nikodym Property ($RNP$) if and only if every closed bounded and convex subset of $X$ has slices with arbitrarily small diameter.  $X$ has the Point of Continuity Property ($PCP$) if every closed bounded and convex subset of $X$  has relatively weakly open subsets with arbitrarily small diameter. $X$ is said to be Strongly Regular ($SR$) if every closed, convex and bounded subset of $X$ has convex combination of slices with  arbitrarily small diameter. For more details, see \cite{B3}, \cite{GGMS} and \cite{GMS}. It is clear then that $RNP$ implies $PCP$ and $PCP$ implies $SR$. It is also well known that none of these implications can be reversed. Clearly, $RNP$ implies $BDP$, $PCP$ implies $BHP$ and $SR$ implies $BSCSP$. The diagram below gives a clear picture.
 $$ RNP \Longrightarrow \quad PCP \Longrightarrow \quad  SR$$ $\quad \quad \quad \quad \quad \quad \quad \quad\quad\quad\quad\quad\quad\quad\quad\quad \Big \Downarrow \quad \quad\quad\quad\quad \Big \Downarrow \quad \quad\quad\quad\quad \Big \Downarrow$  $$ BDP \Longrightarrow  BHP \Longrightarrow  BSCSP$$
  It was proved in \cite{GGMS} that $X$ is strongly regular if and only if every nonempty bounded convex set $K$ in $X$ is contained in norm closure of $SCS(K)$ i.e. $SCS$ points of $K$. Later it was proved in \cite{S} that a Banach space has $RNP$ if and only if it is SR and has the Krein Milman Property (KMP), i.e. every closed bounded convex subset K of $X$ is the norm-closed convex hull of its extreme points. All the three properties discussed in this paper in a way, are ''localised"(to the closed unit ball) versions of the three geometric properties $RNP, PCP$ and $SR$.

    In this work we introduce Ball Huskable Property ($BHP$) and explore its relation with $BSCSP$ and $BDP$. We observe that $BDP$ implies $BHP$ which in turn implies $BSCSP$ and none of the implications can be reversed. We prove certain stability results for $BDP$ , $BHP$ and  $BSCSP$. We further explore these properties in the context of three space property.The spaces that we will be considering have been well studied in literature. A large class of function spaces like the Bloch spaces, Lorentz and Orlicz spaces, spaces of vector valued functions and spaces of compact operators are examples of the spaces we will be considering, for details, see \cite{HWW} 
    
 \section{Stability Results}
 The following result will  be useful in our discussion.

\begin{proposition} \label{A1}
 A Banach space $X$ has $BDP$ (resp. $BHP$ , $BSCSP$) if and only if $X^{**}$ has $w^*$-$BDP$ (resp. $w^*$-$BHP$, $w^*$-$BSCSP$). 
\end{proposition}

\begin{proof}
Suppose $X$ has  $BDP.$
  Let $\varepsilon >0.$ Then there exists a slice $S(B_X, x^*, \alpha)$ of $B_X$ with diameter less than $\frac{\varepsilon}{2}.$\\
Claim:~ $S(B_X, x^*, \alpha)$ is $w^*$ dense in the $w^*$-slice $S(B_{X^{**}}, x^*, \alpha)$ of $B_{X^{**}}.$\\
Indeed, fix $x^{**}$ $\in S(B_{X^{**}}, x^*, \alpha)$. By Goldstine's Theorem  , there is a net $(x_\beta)$ in $B_X$ which converges to $x^{**}$ in the $w^*$-topology. Since, 
$$\displaystyle{\lim_{\beta}}\quad x^*(x_\beta) =  x^{**} (x^*) > 1-  \alpha$$   
 So, there exists $\beta_0$ such that $(x_\beta) \in S(B_X, x^*, \alpha)$ for all $\beta \geqslant \beta_0$. Hence the claim.
 
 Now let $x^{**}$ , $\tilde{x}^{**}$ $\in S(B_{X^{**}}, x^*, \alpha)$. Then there exist net $(x_\beta)$ and $(\tilde{x}_\beta)$ in $S(B_X, x^*, \alpha)$ such that $(x_\beta - \tilde{x}_\beta)$ converges to $x^{**} -\tilde{x}^{**}$ in the  $w^*$ -topology. So, 
$$\Vert x^{**} -\tilde{x}^{**} \Vert\leqslant \displaystyle{\liminf_{\beta}}\quad \Vert x_\beta - \tilde{x}_\beta \Vert\leqslant \frac{\varepsilon}{2} < \varepsilon.$$ 
Thus dia $S(B_{X^{**}}, x^*, \alpha) < \varepsilon$ . Hence $X^{**}$ has $w^*$-BDP.

 Conversely, if $X^{**}$ has $w^*$-$BDP$, it immediately follows that $X$ has $BDP.$
 
 The proofs for $BHP$ and $BSCSP$ follow similarly.
\end{proof}

We immediately have, 
\begin{corollary}
If X has $BDP$ (resp. $BHP$ , $BSCSP$) then $X^{**}$ has  $BDP$ (resp. $BHP$, $BSCSP$) 
\end{corollary}

 The following lemma will be useful.
 
\begin{lemma} \label{p sum lem}
 Let $Z=X\oplus_pY$ , $1\leqslant p < \infty$ , For any $\varepsilon > 0$ and  for any slice there exists a slice $S(B_Z, z^*, \mu)$ of $B_Z$ such that $ S(B_Z, z^*, \mu) \subset S(B_X, x^*, \alpha) \times \varepsilon B_Y.$  Similarly, there exists a slice $S(B_Z, z^*, \gamma)$ of $B_Z$ such that $ S(B_Z, z^*, \gamma) \subset  \varepsilon B_X \times S(B_Y, y^*, \alpha).$
\end{lemma} 

\begin{proof}
Let $S(B_X, x^*, \alpha)$  be any slice of $B_X.$  Also let $\varepsilon > 0.$
 Put $z^* = (x^*,0) \in S_{Z^*}$ . Choose $0< \mu < \alpha$ such that $[1-(1-\mu)^p]^{1/p} < \varepsilon.$  Now, consider a slice of $B_Z$ as , $S(B_Z,z^*,\mu) = \{ z\in B_Z: z^*(z) > 1-\mu \} = \{ z\in B_Z: x^*(x) > 1-\mu \}.$  Then $ S(B_Z,z^*,\mu) \subset S(B_X,x^*,\alpha) \times \varepsilon B_Y$. Indeed, let $z\in S(B_Z,z^*,\mu).$  Then,
 $$1\geqslant \Vert z \Vert^p = \Vert x \Vert^p + \Vert y \Vert^p > (1- \mu)^p + \Vert y \Vert^p$$  Thus, $ \Vert y \Vert^p < 1- (1- \mu)^p$ and so $ \Vert y \Vert < [1- (1- \mu)^p]^{1/p} < \varepsilon$. Also since, $0< \mu < \alpha$, it follows that, $x\in S(B_X,x^*,\alpha)$ . Hence, $z=(x,y)\in S(B_X, x^*, \alpha) \times \varepsilon B_Y.$ The other proof is similar.
\end{proof}

\begin{proposition} \label{BDP p sum}
 Let $X$ and $Y$ be two Banach spaces and $Z=X\oplus_p Y$ , $1\leqslant p < \infty.$  Then $Z$ has $BDP$ if and only if $X$ or $Y$ has $BDP.$
\end{proposition} 

\begin{proof}
Suppose $Z$ has $BDP.$  We prove by contradiction. If possible, let $X$ and $Y$ do not have $BDP.$  Then there exists $\varepsilon > 0$ such that every slice of $B_X$ and $B_Y$ has diameter greater than $\varepsilon.$  Since $Z$ has $BDP$, there exists a slice $S(B_Z, z^* ,\alpha)$ of $B_Z$ with diameter less than $\varepsilon.$\\
Case~1 : $x^*=0$ or $y^*=0$ \\
Without loss of generality, let $y^*= 0.$ Then $x^*\in S_{X^*}.$ Then $S(B_X, x^*, \alpha) \times \{0\}\ \subset S(B_Z, z^*,\alpha)$. 
Thus, 
$$dia S(B_X,x^*,\alpha) = dia(S(B_X,x^*,\alpha) \times \{0\}) \leqslant dia S(B_Z,z^*,\alpha) < \varepsilon, $$
a contradiction.\\

Case-2 : $x^* \neq 0$ and $y^* \neq 0.$ 
Choose $z_0=(x_0,y_0) \in S(B_Z,z^*,\frac{\alpha}{4})$ with $\Vert z_0 \Vert = 1$. Now, $S(B_X,\frac{x^*}{\Vert x^* \Vert},\frac{\alpha}{2})$ and $S(B_Y,\frac{y^*}{\Vert y^* \Vert},\frac{\alpha}{2})$ are slices of $B_X$ and $B_Y$ respectively. Hence, $dia S(B_X,\frac{x^*}{\Vert x^* \Vert},\frac{\alpha}{2}) >\varepsilon$ and $dia S(B_Y,\frac{y^*}{\Vert y^* \Vert},\frac{\alpha}{2})>\varepsilon$. There exists $x,\tilde{x} \in S(B_X,\frac{x^*}{\Vert x^* \Vert},\frac{\alpha}{2})$ and $y,\tilde{y} \in S(B_Y,\frac{y^*}{\Vert y^* \Vert},\frac{\alpha}{2})$  such that $\Vert x - \tilde{x} \Vert > \varepsilon$ and   $\Vert y - \tilde{y} \Vert > \varepsilon.$  \\
Let  $z=(\Vert x_0\Vert x , \Vert y_0\Vert y)$ and $\tilde{z}=(\Vert x_0\Vert \tilde{x} , \Vert y_0\Vert \tilde{y}).$ Clearly, $ z, \tilde{z} \in S(B_Z,z^*,\alpha).$ 
 Also, $\Vert z - \tilde{z} \Vert ^p  = \Vert x_0 \Vert ^p \Vert x - \tilde{x} \Vert^p + \Vert y_0 \Vert ^p \Vert y - \tilde{y} \Vert^p  > \varepsilon^p (\Vert x_0 \Vert ^p + \Vert y_0 \Vert ^p)  = \varepsilon ^p$ which implies $\Vert z - \tilde{z} \Vert > \varepsilon,$ a contradiction.\\
 Hence either $X$ or $Y$ has $BDP.$

Conversely, assume that either $X$ or $Y$ has $BDP$. Without loss of generality let $X$ have $BDP.$ Let $\varepsilon > 0.$ Then there exists a slice $S(B_X,x^*,\alpha)$ of $B_X$ with diameter $<\varepsilon,$ where $x^* \in S_{X^*}$ and $\alpha>0.$ From Lemma $\ref{p sum lem},$ there exists a slice $S(B_Z,z^*,\mu)$ of $B_Z$ such that $ S(B_Z,z^*,\mu) \subset S(B_X,x^*,\alpha) \times \varepsilon B_Y$
Consequently,$$ dia (B_Z,z^*,\mu)\leqslant dia S(B_X,x^*,\alpha) + dia  (\varepsilon B_Y)< \varepsilon + 2\varepsilon = 3\varepsilon$$
\end{proof}

\begin{corollary}
Let $X=\oplus_p X_i.$ If $X_i$ has $BDP$ for some i , then $X$ has $BDP.$ 

\end{corollary}
We quote the following Lemma from \cite{L1}.

\begin{lemma} \cite{L1}\label{BDP lem inf}
 Let $Z= X\oplus_{\infty} Y$ then for every slice $S(B_Z,z^*,\alpha)$ of $B_Z$ there exists a slice $S(B_X,x^*,\mu_1)$ of $B_X$ , a slice $S(B_Y,y^*,\mu_2)$ of $B_Y,$  $x_0 \in B_X$ and  $y_0 \in B_Y$ such that $S(B_X,x^*,\mu_1) \times \{y_0\} \subset S(B_Z,z^*,\alpha)$ and $\{x_0\} \times S(B_Y,y^*,\mu_2)  \subset S(B_Z,z^*,\alpha).$ 
\end{lemma}

\begin{proposition} \label{max sum}
  $Z= X\oplus_{\infty} Y$ has $BDP$ if and only if both $X$ and $Y$ have $BDP.$
\end{proposition}

\begin{proof}
First suppose that $Z$ has $BDP.$ Let $0< \varepsilon <2.$ Then there exists a slice $S(B_Z,z^*,\alpha)$ of $B_Z,$ such that $ dia (S(B_Z,z^*,\alpha) <\varepsilon,$ where $z^*=(x^*,y^*) \in S_{Z^*}$ and $\alpha>0.$\\
Claim :$x^* \neq 0$ and $y^* \neq0$. \\
If not, let $x^* =0.$ Then $\Vert y^* \Vert = 1.$ Choose any fixed $y_0 \in S(B_Y, y^*, \alpha).$ Then $B_X \times \{y_0\} \subset S(B_Z, z^*, \alpha).$
 So, $ 2 = dia (B_X \times \{y_0\})\leqslant dia S(B_Z, z^*, \alpha) <\varepsilon,$ a contradiction. Hence the claim.
 Now from Lemma $\ref{BDP lem inf},$ there exists a slice $S(B_X,x^*,\mu)$ of $B_X$ and $y_1 \in B_Y$ such that $S(B_X,x^*,\mu) \times \{y_1\} \subset S(B_Z,z^*,\alpha).$ 
 Consequently, $dia  S(B_X,x^*,\mu) = dia  (S(B_X,x^*,\mu)\times \{y_1\} )\leqslant diaS (B_Z, z^*, \alpha) < \varepsilon.$ 
 Thus, $X$ has $BDP.$ Similarly, $Y$ has $BDP$.\\
 Conversely let both $X$ and $Y$ have $BDP$ and $\varepsilon>0$.  So, there exist slices $S(B_X,x^*,\alpha_1)$ and $S(B_Y,y^*,\alpha_2)$ of $B_X$ and $B_Y$ respectively such that $ dia  S(B_X,x^*,\alpha_1)<\varepsilon$ and $ dia  S(B_Y,y^*,\alpha_2)<\varepsilon.$ Choose $0<\gamma<\min\{\alpha_1,\alpha_2\}.$ Consider slice $S(B_Z,z^*,\gamma)$ of $B_Z$ such that $z^*=(\frac{x^*}{2},\frac{y^*}{2})$. Then $\Vert z^*\Vert =1.$ Then $S(B_Z,z^*,\gamma)\subset S(B_X,x^*,\alpha_1) \oplus_{\infty} S(B_Y,y^*,\alpha_2). $ Indeed, let $z=(x,y)\in S(B_Z,z^*,\gamma).$ Then $$z^*(z)= \frac{x^*}{2}(x)+\frac{y^*}{2}(y)>1-\gamma$$
 $$\Rightarrow 1+y^*(y)\geqslant x^*(x)+y^*(y)>2-2\gamma$$
 $$\Rightarrow y^*(y)>1-2\gamma>1-\alpha_2$$
 Thus $y\in S(B_Y,y^*,\alpha_2)$ and similarly $x\in S(B_X,x^*,\alpha_1).$ Finally, $dia  S(B_Z,z^*,\gamma)\leqslant \varepsilon$ as both $S(B_X,x^*,\alpha_1)$ and $S(B_Y,y^*,\alpha_2)$ are of diameter $<\varepsilon.$  
\end{proof}

\begin{proposition}
 Let $X$ and $Y$ be two Banach spaces and $Z=X\oplus_pY$ , $1\leqslant p < \infty$ . Then $Z$ has $BHP$ if and only if $X$ or $Y$ has $BHP$.
\end{proposition} 

\begin{proof}
Suppose $Z$ has $BHP$. If possible let $X$ and $Y$ fail $BHP.$ Then there exists $\varepsilon > 0$ such that every relatively weakly open subset of $B_X$ and $B_Y$ has diameter greater than  $\varepsilon.$  Now since $Z$ has $BHP$ so there exists a relatively weakly open subset $W$ of $B_Z$ with diameter less than $\varepsilon.$  Fix $z_0=(x_0,y_0) \in W \bigcap S_Z$ . Then there exists a basic weakly open subset,
$W_0=\{z\in B_Z : \vert z_i^*(z-z_0)\vert <1 , i=1,2,...n\} \subset W$ where $z_i^* =(x_i^*,y_i^*)$, i = 1,2,...n. We consider two cases.\\
Case~1 : $x_0=0$ or $y_0=0$ \\
without loss of generality let $y_0= 0.$ Thus $x_0 \in S_X.$  Then \\
$U=\{ x\in B_X : \vert x_i^*(x-x_0)\vert <1 ; i=1,2,...n\}$ is nonempty relatively weakly open subset of $B_X$ . By our assumption $dia (U)$ $>\varepsilon$ . So there exists $x,\tilde{x} \in U$ such that $\Vert x-\tilde{x} \Vert > \varepsilon.$  Now, $z=(x,0)$ and $\tilde{z}=(\tilde{x},0)$ are in $W_0$  and $\Vert z- \tilde{z} \Vert = \Vert x- \tilde{x} \Vert > \varepsilon$, a contradiction.\\
Case~2 : $x_0 \neq 0$ and $y_0 \neq 0$ \\
Consider, $U=\{x \in B_X : \vert x_i^*(x-\frac{x_0}{\Vert x_0\Vert}) \vert < \frac{1}{2\Vert x_0\Vert} ; i=1,2,..n\}$ \\and 
$V=\{y \in B_Y : \vert y_i^*(y-\frac{y_0}{\Vert y_0\Vert}) \vert < \frac{1}{2\Vert y_0\Vert} ; i=1,2,..n\}$ \\
Then $U$ and $V$ are nonempty relatively weakly open subsets of $B_X$ and $B_Y$ respectively and so $dia (U)>$ $\varepsilon$ and $dia (V)>$ $\varepsilon$ . Hence, there exists $x,\tilde{x} \in U$ and $y,\tilde{y} \in V$ such that $\Vert x- \tilde{x} \Vert > \varepsilon$ and $\Vert y- \tilde{y} \Vert > \varepsilon.$ \\
Thus $z=(\Vert x_0\Vert x , \Vert y_0\Vert y)$ and $\tilde{z}=(\Vert x_0\Vert \tilde{x} , \Vert y_0\Vert \tilde{y})$ are in $W.$ \\
Indeed, $\Vert z \Vert^p = \Vert x_0 \Vert^p \Vert x \Vert^p + \Vert y_0 \Vert^p \Vert y \Vert^p \leqslant \Vert x_0 \Vert^p + \Vert y_0 \Vert^p = 1$,  and \\
$\forall i = 1,2,....n$ we have, \\
$\vert z_i^*(z-z_0)\vert = \vert x_i^*(\Vert x_0\Vert x - x_0) + y_i^*(\Vert y_0\Vert y - y_0)\vert$ $\leqslant \Vert x_0 \Vert \vert x_i^*(x-\frac{x_0}{\Vert x_0\Vert}) \vert +  \Vert y_0 \Vert \vert y_i^*(y-\frac{y_0}{\Vert y_0\Vert}) \vert$ 
$$\quad \quad \quad \quad \quad \quad \quad \quad \quad \quad \quad \quad\quad \quad  < \Vert x_0 \Vert \frac{1}{2\Vert x_0\Vert} + \Vert y_0 \Vert \frac{1}{2\Vert y_0\Vert} =1$$ 
Similarly for $\tilde{z}.$  \\
Finally,
 $\Vert z - \tilde{z} \Vert ^p $ = $\Vert x_0 \Vert ^p \Vert x - \tilde{x} \Vert^p + \Vert y_0 \Vert ^p \Vert y - \tilde{y} \Vert^p $ $> \varepsilon^p (\Vert x_0 \Vert ^p + \Vert y_0 \Vert ^p) $ = $\varepsilon ^p$ \\ and so $\Vert z - \tilde{z} \Vert > \varepsilon$, a contradiction.\\
 Hence either $X$ or $Y$ has $BHP.$ \\
Conversely assume that either $X$ or $Y$ has $BHP.$  Without loss of generality, suppose  $X$ has $BHP.$  Let $\varepsilon > 0.$ Then there exists a relatively weakly open set $W$ of $B_X$ with diameter $<\varepsilon.$  Since slices of $B_X$ forms a subbase for relatively weakly open subset of $B_X$, so there exists slices $S(B_X,x_i^*,\alpha_i)$, i= 1,2,...n of $B_X$ such that $\bigcap _{i=1}^{n} S(B_X,x_i^*,\alpha_i) \subset W.$   Now from Lemma $\ref{p sum lem}$ for each i , we get a slice $S(B_Z,z_i^*,\mu_i)$ of $B_Z$ such that  $S(B_Z,z_i^*,\mu_i) \subset S(B_X,x_i^*,\alpha_i) \times \varepsilon B_Y$.  Choose $\mu < \min\{\mu_1,\mu_2,...\mu_n\}.$ \\
Thus , $\bigcap _{i=1}^{n} S(B_Z,z_i^*,\mu) \subset \bigcap_{i=1}^{n} S(B_X,x_i^*,\alpha_i) \times \varepsilon B_Y \subset W \times \varepsilon B_Y.$\\
Hence $T = \bigcap _{i=1}^{n} S(B_Z,z_i^*,\mu)$ is a relatively weakly open subset of $B_Z$  with diameter less than $3\varepsilon.$  Consequently $Z$ has $BHP$. 

\end{proof}

\begin{corollary}
Let $X=\oplus_p X_i$ . If $X_i$ has $BHP$ for some i , then $X$ has $BHP$.

\end{corollary}

The following Lemma from \cite{ALN} will be useful.
\begin{lemma} \label {BHP inf sum} \cite{ALN}
Let $X$ and $Y$ be Banach spaces and $W$ be a nonempty relatively weakly open subset in unit ball of $Z= X \oplus _{\infty} Y.$  Then $U$ and $V$ can be chosen to be relatively weakly opn subsets of $B_X$ and $B_Y$ respectively such that $U \times V \subset W.$
\end{lemma}

\begin{proposition}\label{l-infinity bhp}
 $Z= X\oplus_{\infty} Y$ has $BHP$ if and only if both $X$ and $Y$ have $BHP$ .
\end{proposition}

\begin{proof}
First suppose that $Z$ has $BHP$ . Also let $\varepsilon >0$. Then $B_Z$ has a relatively weakly open subset $W_0$ with diameter less than $<\varepsilon.$ Then by Lemma $\ref{BHP inf sum},$ there exists a relatively weakly open subset $U_0$ in $B_X$ and $V_0$ in $B_Y$ such that $U_0 \times V_0 \subset W_0.$  Fix $u_0\in U_0$ and $v_0\in V_0$ . Then,\\
dia $(U_0)$ = dia $(U_0 \times \{v_0\})\leqslant$ dia ($U_0 \times V_0) \leqslant$ dia $W_0 < \varepsilon.$ \\
Similarly for $V_0.$  Consequently, both $X$ and$Y$ have $BHP$ . \\
Conversely suppose  $X$ and $Y$ have $BHP$ and $\varepsilon>0$. So, there exists relatively weakly open subset $U$ and $V$ of $B_X$ and $B_Y$ respectively such that $dia (U)<\varepsilon$ and $dia (V)<\varepsilon$. Since slices of $B_X$ forms a subbase for relatively weakly open subset of $B_X$ , so there exists slices $S(B_X,x_i^*,\alpha_i)$ , i= 1,2,$\ldots$, n of $B_X$ such that $\bigcap _{i=1}^{n} S(B_X,x_i^*,\alpha_i) \subset U$ and similarly  there exists slices $S(B_Y,y_i^*,\beta_j)$ , j= 1,2,$\ldots$, m of $B_Y$ such that $\bigcap _{j=1}^{m} S(B_Y,y_i^*,\beta_j) \subset V.$ Without loss of generality let $n\geqslant m$. Then proceeding same way as in Proposition $\ref{max sum}$ we get slices of $B_Z$, $S(B_Z,z_i^*,\gamma_i)\quad \forall i=1,\ldots,n$ \\where
 $\gamma_i<\min\{\alpha_i,\beta_i\}\quad if\quad  i=1,\ldots,m$ and $\gamma_i<\min\{\alpha_i,\beta_m\}\quad if\quad i=m+1,\ldots,n$ such that 
$$S(B_Z,z_i^*,\gamma_i)\subset S(B_X,x_i^*,\alpha_i) \oplus_{\infty} S(B_Y,y_i^*,\beta_i) \quad\forall i=1,\ldots,m$$ and
 $$S(B_Z,z_i^*,\gamma_i)\subset S(B_X,x_i^*,\alpha_i) \oplus_{\infty} S(B_Y,y_i^*,\beta_i) \quad\forall i=m+1,\ldots,n$$
 Thus , $\bigcap_{i=1}^{n} S(B_Z,z_i^*,\gamma_i)$ $\subset$
  $ \bigcap_{i=1}^{n} S(B_X,x_i^*,\alpha_i) \oplus_{\infty} \bigcap_{i=1}^{m} S(B_Y,y_i^*,\beta_i)$ 
  $\subset$ 
   $ U\oplus_{\infty} V .$\\
    Hence, $dia (\bigcap_{i=1}^{n} S(B_Z,z_i^*,\gamma_i))<\varepsilon$
\end{proof}

\begin{proposition}
 Let $X$ and $Y$ be two Banach spaces and $Z=X\oplus_pY$ , $1\leqslant p < \infty.$ $Z$ has $BSCSP$ if $X$ or $Y$ has $BSCSP.$
\end{proposition} 

\begin{proof}
Without loss of generality, let $X$ have $BSCSP.$  Let $\varepsilon > 0.$ Then there exists a convex combination of slices $\sum_{i=1}^{n} \lambda_i S(B_X, x_i^*, \alpha_i)$ , $\lambda_i \geqslant 0 , \sum_{i=1}^{n} \lambda_i =1 $ of $B_X$ with diameter less than $\varepsilon.$ By Lemma $\ref{p sum lem},$ for each i, there exists a  slice $S(B_Z,z_i^*,\mu_i)$ of $B_Z$ such that  $S(B_Z,z_i^*,\mu_i) \subset S(B_X,x_i^*,\alpha_i) \times \varepsilon B_Y.$  
Thus, $\sum _{i=1}^{n} \lambda_i S(B_Z,z_i^*,\mu_{i}) \subset \sum_{i=1}^{n} \lambda_i \Big[ S(B_X,x_i^*,\alpha_i) \times \varepsilon B_Y \Big].$ 
Hence, $dia(\sum _{i=1}^{n} \lambda_i S(B_Z,z_i^*,\mu_{i})) < 3 \epsilon.$ Consequently $Z$ has $BSCSP.$ 

\end{proof}

\begin{corollary}
Let $X=\oplus_p X_i.$  If $X_i$ has $BSCSP$ for some $i,$ then $X$ has $BSCSP.$ 
\end{corollary}

\begin{proposition}
	If $Z= X\oplus_1 Y$ has $BSCSP,$ then either $X$ or $Y$ has $BSCSP.$ 
\end{proposition}

\begin{proof}
	If possible, let  $X$ and  $Y$ fail $BSCSP.$ Then there exists $\varepsilon > 0$ such that every convex combination of slices of $B_X$ and $B_Y$ have diameter greater than  $\varepsilon.$ Now since $Z$ has $BSCSP,$ so there exists convex combination of slices $\sum_{i=1}^{n} \lambda_i S(B_Z,z_i^*,\alpha_i)$  of $B_Z$ with diameter  less than $\varepsilon.$
	Observe,$ 1 = \Vert z_i^* \Vert = \max \{\Vert x_i^* \Vert , \Vert y_i^* \Vert \} ,\quad  i= 1,2,......n.$ 
	We consider two disjoint subsets $I$ and $J$ where,
	$ I = \{ i : \Vert x_i^* \Vert = 1 \}\quad and \quad  J = \{ j  : \Vert y_j^* \Vert = 1 \}$
	Now, $ S(B_X,x_i^*,\alpha_i) \times \{0\} \subset S(B_Z,z_i^*,\alpha_i) \quad \forall i \in I$
	and $ \{0\} \times S(B_Y,y_j^*,\alpha_j)  \subset S(B_Z,z_j^*,\alpha_j) \quad \forall j \in J.$\\
	Let $\quad \lambda_I =\sum_{i\in I} \lambda_i \quad and \quad \lambda_J =\sum_{j\in J} \lambda_j $
	
	Case-1 : $ \lambda_I = 0 $ or $ \lambda_J = 0 $\\
	Without loss of generality, let $ \lambda_I = 0 $\\ 
	Then $ \lambda_J = 1 $ and so $\sum_{j\in J} \lambda_j S(B_Y,y_j^*,\alpha_j)$ is a convex combination of slices of $B_Y,$ hence 
	dia ($\sum_{j\in J} \lambda_j S(B_Y,y_j^*,\alpha_j)) > \varepsilon $
	So, there exist $y, \tilde{y} \in \sum_{j\in J} \lambda_j S(B_Y,y_j^*,\alpha_j)$ such that $\Vert y - \tilde{y} \Vert > \varepsilon.$
	Hence, $\Vert (0,y) - (0,\tilde{y}) \Vert > \varepsilon$, which is a contradiction.
	
	Case-2 : $ \lambda_I \neq 0 $ or $ \lambda_J \neq 0 $\\
	So we have, $ \sum_{i\in I} \frac{\lambda_i}{\lambda_I} S(B_X, x_i^* , \alpha_i) \times \{0\} \subset  \sum_{i\in I} \frac{\lambda_i}{\lambda_I} S(B_Z, z_i^* , \alpha_i)$ and $\{0\} \times \sum_{j\in J} \frac{\lambda_j}{\lambda_J} S(B_Y, y_j^* , \alpha_j)   \subset  \sum_{j\in J} \frac{\lambda_j}{\lambda_J} S(B_Z, z_j^* , \alpha_j).$
	Again, $dia (\sum_{i\in I} \frac{\lambda_i}{\lambda_I} S(B_X, x_i^* , \alpha_i)) > \varepsilon $ and $dia (\sum_{j\in J} \frac{\lambda_j}{\lambda_J} S(B_Y, y_j^* , \alpha_j)) > \varepsilon.$
	So, there exist \ $x,\tilde{x} \in \sum_{i\in I} \frac{\lambda_i}{\lambda_I} S(B_X, x_i^* , \alpha_i)$ and \ $y,\tilde{y} \in \sum_{j\in J} \frac{\lambda_j}{\lambda_J} S(B_Y, y_j^* , \alpha_j)$ such that $ \Vert x-\tilde{x} \Vert > \varepsilon$ and $\Vert y-\tilde{y} \Vert > \varepsilon$

	Observe,  $$ (\lambda_I x, \lambda_J y) = (\lambda_I x, 0) +(0, \lambda_J y) \in \sum_{i\in I} \lambda_i S(B_X, x_i^* , \alpha_i)\times \{0\} + \{0\} \times \sum_{j\in J} \lambda_j S(B_Y, y_j^* , \alpha_j)$$ 
	$$\hspace{3.1 cm} \subset \sum_{i\in I} \lambda_i S(B_Z, z_i^* , \alpha_i)+\sum_{j\in J} \lambda_j S(B_Z, z_j^* , \alpha_j)$$  $$= \sum_{i=1}^{n} \lambda_i S(B_Z, z_i^* , \alpha_i) $$
	Similarly, $(\lambda_I \tilde{x} , \lambda_J \tilde{y} ) \in \sum_{i=1}^{n} \lambda_i S(B_Z, z_i^* , \alpha_i) $\\
	Also, $\Vert (\lambda_I x, \lambda_J y) - (\lambda_I \tilde{x} , \lambda_J \tilde{y} ) \Vert = \Vert (\lambda_I (x-\tilde{x}) , \lambda_J (y-\tilde{y} ) \Vert = \lambda_I \Vert x-\tilde{x} \Vert + \lambda_J \Vert y-\tilde{y} \Vert > \varepsilon (\lambda_I + \lambda_J) = \varepsilon, $ a contradiction .
	Hence, either $X$ or $Y$ has $BSCSP.$
	
\end{proof}

\begin{proposition}\label{l-infinity bscsp}
 If $Z= X\oplus_{\infty} Y$ has $BSCSP$ , then both $X$ and $Y$ have $BSCSP.$ 
\end{proposition}

\begin{proof}
Suppose  $Z$ has $BSCSP.$ So, for any $\varepsilon> 0,$ there exists a convex combination of slices $\sum_{i=1}^{n} \lambda_i S(B_Z, z_i^*, \alpha_{i}), \lambda_i \geqslant 0 , \sum_{i=1}^{n} \lambda_i =1 $  of $B_Z$ such that $ dia (\sum_{i=1}^{n} \lambda_i S(B_Z, z_i^*,\alpha_i) <\varepsilon.$ By Lemma $\ref{BDP lem inf},$ for each i , there exits a slice $S(B_X, x_i^*, \mu_i)$ of $B_X$ and $y_i \in B_Y$ such that  $S(B_X ,x_i^*, \mu_i) \times \{y_i\} \subset S(B_Z, z_i^*, \alpha_i).$
Let $y_0 = \sum_{i=1}^{n} \lambda_i y_i.$
$\Big[\sum_{i=1}^{n} \lambda_i S(B_X, x_i^*,\mu_i)\Big] \times \{y_0\} = \sum_{i=1}^{n} \lambda_i \Big[ S(B_X, x_i^*, \mu_i) \times \{y_i\} \Big] \subset \sum_{i=1}^{n} \lambda_i S(B_Z, z_i^*, \alpha_i).$ 
Hence,  $dia(\sum_{i=1}^{n} \lambda_i S(B_X, x_i^*, \mu_i)) = dia(\Big[\sum_{i=1}^{n} \lambda_i S(B_X, x_i^*, \mu_i)\Big] \times \{y_0\}$) $\leqslant dia (\sum_{i=1}^{n} \lambda_i S(B_Z, z_i^*, \alpha_i) \leq  \varepsilon.$
So, $X$ has $BSCSP.$ Similarly for $Y.$
\end{proof}

Similar results are true for $w^*$-versions. We omit the proofs which are similar. 



\begin{proposition} \label{w star}
Let $X$ and $Y$ be Banach spaces and $Z=X\oplus_p Y$ , $1< p \leqslant \infty.$  Then 
\begin{enumerate}
\item $Z^*$ has $w^*BHP$ ($w^*BDP$) if and only if $X^*$ or $Y^*$ has $w^*BHP$ ($w^*BDP$).
\item  $X^*$ or $Y^*$ has $w^*BSCSP$ implies $Z^*$ has $w^*BSCSP.$ 
\end{enumerate}

\end{proposition} 

\begin{proposition}
Let $X$ and $Y$ be Banach spaces and $Z=X\oplus_1 Y$,   Then\\
$Z^*$ has $w^*BHP$ ( resp.  $w^*BDP$ , $w^*BSCSP$ ) implies both $X^*$ and $Y^*$ has $w^*BHP$ ( resp. $w^*BDP$ , $w^*BSCSP$).
\end{proposition}

\begin{proposition}
	Let  $X = \oplus_{c_0(\Gamma)}X_i$. If  $X$  has $BDP$ (resp. $BHP$ , $BSCSP$), then each of $X_i$ has $BDP$ (resp. $BHP$ , $BSCSP$) .
\end{proposition}

\begin{proof}
	Fix i. Note that $ X = \oplus_{c_0(\Gamma)}X_j$ = $X_i\oplus_{\infty} \Big(\oplus_{c_0(\Gamma \setminus \{i\})} X_j\Big).$
	The rest follows from Propositions \ref{max sum}, \ref{l-infinity bhp} and \ref {l-infinity bscsp}.
	 \end{proof}
\begin{remark}
	The converse of the above Proposition may not be true.
	 Indeed, $\mathbb{R}$ has $BDP$ (resp. $BHP$ , $BSCSP$) but $c_0$ does not have $BDP$ (resp. $BHP$ , $BSCSP$).
\end{remark}

We now show that none of the implications  in the following  diagram can be reversed.

$$ BDP \Longrightarrow \quad BHP \Longrightarrow \quad  BSCSP$$ $\quad \quad \quad \quad \quad \quad \quad \quad\quad\quad\quad\quad\quad\quad\quad\quad \Big \Uparrow \quad \quad\quad\quad\quad \Big \Uparrow \quad \quad\quad\quad\quad \Big \Uparrow$  $$ w^*BDP \Longrightarrow  w^*BHP \Longrightarrow  w^*BSCSP$$


\begin{example} \label{example1}
	\blr
\item $ BHP \nRightarrow BDP.$\\ It was proved in \cite{BGLPRZ1}  that, if a Banach space $X$ contains an isomorphic copy of $c_0$ then it can be equivalently renormed, so that every slice of unit ball of $X$ has diameter two but it has a relatively weakly open subset of arbitrarily small diameter.
Hence a Banach space containing an isomorphic copy of $c_0$ can be equivalently renormed so that it has $BHP$ but not $BDP.$ 
\item $ BSCSP \nRightarrow BHP.$\\
It was proved in \cite{BGLPRZ2} that , if a Banach space $X$ contains an isomorphic copy of $c_0$ then it can be equivalently renormed, so that every relatively weakly open subset of unit ball of $X$ has diameter two but it has convex combination of slices of arbitrarily small diameter.
Hence a Banach space containing an isomorphic copy of $c_0$ can be equivalently renormed so that it has $BSCSP$ but not $BHP.$
  
\item $ w^*BHP \nRightarrow  w^*BDP.$ \\ If we consider bidual of the  space in (i), then from Proposition $\ref{A1}$,  we get a space with $w^*BHP$  but not $w^* BDP$ .

\item $w^* BSCSP \nRightarrow w ^* BHP.$ \\ If we consider bidual of the  space in (ii), then from Proposition $\ref{A1}$, we get a space with $w^*BSCSP$  but not $w^* BHP.$

\item ${BDP \nRightarrow w^*BDP}$, $ {BHP \nRightarrow w^*BHP}$, $ {BSCSP \nRightarrow w^*BSCSP}$\\
Let $X=C[0,1].$ It is well known that $X^*=L_1[0,1] \oplus _1 Z$, for some subspace $Z$ of $X^*$ with $RNP$ and hence $Z$ has $BDP.$ Since  $Z$ has $BDP,$ it follows from Proposition \ref{BDP p sum}, $X^*$ has $BDP$ and hence it has $BHP$ and $BSCSP.$ But every convex combination of $w^*$ slices of $B_{X^*}$ has diameter two ( 
see  \cite{BGLPRZ3} for details). Hence, $C[0,1]^*$ does not have $w^*$-$BSCSP$ and consequently it does not have  $w^*$-$BDP$ , $w^*$-$BHP.$
 \el
\end{example}
\begin{remark}
In recent times,the study of diameter two properties has been a very active area of research in the geometry of Banach spaces( see \cite{BGLPRZ1},\cite{BGLPRZ2},\cite{BGLPRZ3},\cite{ALN},\cite{L1}). As is evident from the examples above that a Banach space may have one version of small diameter property and simultaneously may have a version of diameter two property. More investigations about the geometric implications of these observations will be an interesting topic of research in future. 	
\end{remark}

\section{Lebesgue Bochner Function Spaces}
Let $(\Omega,\mathcal{A},\mu)$ be a finite measure space. Then for $1\leqslant p<\infty,$  
$L^p(\mu,X)$ denote the Banach space of all Lebesgue-Bochner 
function of p-integrable $X$ valued functions defined on $\Omega$  with  norm $\|f\|_p=\Big(\int_{\Omega} \|f(t)\|^p d\mu (t)\Big)^{1/p}$ ( see \cite{DU} for details).
\begin{proposition} \label{bochner bhp}
	Let $(\Omega,\mathcal{A},\mu)$ be a finite measure space.Let $1\leqslant p<\infty.$ Then $L^p(\mu,X)$ has $BHP$ implies $X$ has $BHP.$ 
\end{proposition}
\begin{proof}
Suppose  $X$ does not have  $BHP.$ Then there exists $\varepsilon >0$ such that every relatively weakly open subset of $B_X$ has diameter $>\varepsilon.$ Since $L^p(\mu,X)$ has $BHP,$there exists  $V,$ a weakly open subset of $L^p(\mu,X)$ such that $V\bigcap B_{L^p(\mu,X)}\neq \emptyset$ and $dia(V\bigcap B_{L^p(\mu,X)})<\varepsilon.$ If $X$ is finite dimensional then it has $RNP$ and so it has $BHP.$ So we consider the case when $X$ is infinite dimensional .
	Then $L^p(\mu,X)$ is also infinite dimensional and so $V\bigcap S_{L^p(\mu,X)}\neq \emptyset.$ Choose $f_0\in V\bigcap S_{L^p(\mu,X)}.$ Since simple functions are dense in $L^p(\mu,X)$ so let $f_0=\sum_{i=1}^{n} x_i \chi_{A_i}$ where $A_i(i=1,2,\ldots,n)$ are disjoint sets in $\mathcal{A}$ and $x_i(\neq 0)\in X$  $\forall i=1,2,\ldots,n.$ Since weak topology on $L^p(\mu,X)$ is linear so $(x,y)\mapsto x+y$ from $L^p(\mu,X) \times L^p(\mu,X) \rightarrow L^p(\mu,X)$ is weak continuous.
	Thus there exist weakly open subsets $V_i$ containing $x_i\chi_{A_i}$ in $L^p(\mu,X)$ such that
	\begin{equation}\label{1equ}
		\sum_{i=1}^{n} V_i\subset V
	\end{equation}
	
	Since the map $x\mapsto x \chi_{A_i}$ from $X$ to $L^p(\mu,X)$ is linear and norm continuous so also weak continuous for each i. Thus for each i , there exists weakly open subset $W_i\subset X$ containing $x_i$ such that
	\begin{equation}\label{2equ}
		x \chi_{A_i} \in V_i \quad \forall  x\in W_i 
	\end{equation}
	
	Note that $\frac{W_i}{\|x_i\|}$ is a weakly open set in $X$ and $\frac{W_i}{\|x_i\|}\bigcap B_X\neq\emptyset$ as $\frac{x_i}{\|x_i\|}\in \frac{W_i}{\|x_i\|}\bigcap B_X.$ Then $diam(\frac{W_i}{\|x_i\|}\bigcap B_X)>\varepsilon$ implies $diam({W_i}\bigcap {\|x_i\|}B_X)>\varepsilon {\|x_i\|}.$ Then for each i , there exist $y_i,z_i\in {W_i}\bigcap {\|x_i\|}B_X$ such that $\|y_i-z_i\|>\varepsilon \|x_i\|.$
	 Consider $g=\sum_{i=1}^{n} y_i \chi_{A_i}$ and $h=\sum_{i=1}^{n} z_i \chi_{A_i}.$ 
	 From $\ref{1equ}$ and $\ref{2equ}$  we get $g,h\in V.$ Also $g,h\in B_{L^p(\mu,X)}.$ 
	 Indeed , $\|g\|_p^p=\sum_{i=1}^{n} \|y_i\|^p \mu({A_i})\leqslant \sum_{i=1}^{n} \|x_i\|^p \mu({A_i})=1.$ We argue simialry for $h.$ Finally, \\
	$dia (V\bigcap B_{L^p(\mu,X)})\geqslant \|g-h\|_p=\Big(\sum_{i=1}^{n} \|y_i-z_i\|^p \mu({A_i})\Big)^{1/p}>\Big(\sum_{i=1}^{n} \varepsilon^p \|x_i\|^p \mu({A_i})\Big)^{1/p}=\varepsilon$ , which is a contradiction.
\end{proof}

\begin{proposition} \label{bochner bdp}
	Let $(\Omega,\mathcal{A},\mu)$ be a finite measure space. Let $1\leqslant p<\infty.$ Then $L^p(\mu,X)$ has $BDP$ implies $X$ has $BDP.$
	  
\end{proposition}
\begin{proof}
	Suppose $X$ does not have  $BDP.$ 
	Then there exists $\varepsilon >0$ such that any slice of $B_X$ has diameter $>\varepsilon.$ Since $L^p(\mu,X)$ has $BDP$ so there exists a slice $S(B_{L^p(\mu,X)},f^*,\alpha)$ of $B_{L^p(\mu,X)}$ with diameter $<\varepsilon.$
	Choose $f\in S(B_{L^p(\mu,X)},f^*,\alpha)$ with $\|f\|_p=1.$ 
	Since simple functions are dense in $L^p(\mu,X)$ space , so without loss of generality we can assume that $f=\sum_{i=1}^{n} x_i \chi_{A_i}$ where $A_i(i=1,2,\ldots,n)$ are disjoint sets in $\mathcal{A}$ and $x_i(\neq 0) \in X$  $\forall i=1,2,\ldots,n.$ 
	For each $i=1,\ldots,n$ we define $x_i^*(x)=f^*(x\chi_{A_i}).$ Then clearly $x_i^*$ is linear and bounded.
	 Now for each i , consider the slice $S(\|x_i\|B_X, x_i^*, \|x_i^*\| \|x_i\| - x_i^*(x_i) +\alpha_i)$ 
	$ =\{y\in \|x_i\|B_X : x_i^*(y)>x_i^*(x_i)-\alpha_i\}$ where $\alpha_i>0$ are such that $\sum_{i=1}^{n} \alpha_i \leqslant f^*(f)-(1-\alpha).$ 
	Since any slice of $B_X$ has diameter $>\varepsilon$ so 
	$dia(S(\|x_i\|B_X, x_i^*, \|x_i^*\| \|x_i\| - x_i^*(x_i) +\alpha_i)) >\|x_i\| \varepsilon$ 
	Thus for each i , there exist $y_i,z_i\in S(\|x_i\|B_X, x_i^*, \|x_i^*\| \|x_i\| - x_i^*(x_i) +\alpha_i)$ such that $\|y_i-z_i\|>\|x_i\|\varepsilon.$ Define $g=\sum_{i=1}^{n} y_i \chi_{A_i}$ , $h=\sum_{i=1}^{n} z_i \chi_{A_i}.$ 
	$$\|g\|_p^p=\sum_{i=1}^{n} \|y_i\|^p \mu(A_i)\leqslant \sum_{i=1}^{n} \|x_i\|^p \mu(A_i)=1$$
	Similarly for h. Thus $g,h\in B_{L^p(\mu,X)}.$ Also \\
	$f^*(g)=f^*(\sum_{i=1}^{n} y_i \chi_{A_i})
	=\sum_{i=1}^{n} f^*(y_i \chi_{A_i})
	=\sum_{i=1}^{n} x_i^*(y_i)$
	$$>\sum_{i=1}^{n} (x_i^*(x_i)-\alpha_i)$$
	$$=\sum_{i=1}^{n} x_i^*(x_i)-\sum_{i=1}^{n} \alpha_i$$
	$$ = \sum_{i=1}^{n} f^*(x_i\chi_{A_i}) - \sum_{i=1}^{n} \alpha_i $$
	$$ = f^*(f)-\sum_{i=1}^{n} \alpha_i 
	\geqslant 1-\alpha$$
	Thus $g\in S(B_{L^p(\mu,X)},f^*,\alpha).$ Similarly $h \in S(B_{L^p(\mu,X)},f^*,\alpha).$\\
	Hence 
	$$ dia(S(B_{L^p(\mu,X)},f^*,\alpha))\geqslant \|g-h\|_p = \Big(\sum_{i=1}^{n} \|y_i-z_i\|^p \mu({A_i})\Big)^{1/p}>\Big(\sum_{i=1}^{n} \varepsilon^p \|x_i\|^p \mu({A_i})\Big)^{1/p}=\varepsilon$$  which is a contradiction.
\end{proof}

\begin{remark}
	We do not know whether the converse of the Propositions $\ref{bochner bhp}$, $\ref{bochner bdp}$ are true or not in general. However, converse of Proposition \ref{bochner bhp}, fails for $p=1$. Indeed, $\mathbb{R}$ has $BHP$ but $L^1(\mu ,\mathbb{R})$  where $\mu$ is Lebesgue measure on $[0,1]$ (which is simply denoted by $L^1[0,1]$) does not have $BHP.$ 
\end{remark}


\section{Three Space Property}
We quote the following result from \cite{GGMS}
 \begin {proposition}\label{sr}
        \cite{GGMS}  $X$ is strongly regular if and only if every closed convex bounded subset $D$ of $X$ is the norm closure of its $SCS$ points. 
        \end{proposition}
\begin{proposition}
	If $X/Y$ is strongly regular and $Y$ has $BSCSP,$ then $X$ has $BSCSP.$
\end{proposition}
\begin{proof}
	Suppose  $X$ does not have  $BSCSP.$
	 Then there exists $\varepsilon> 0$ such that diameter of any convex combination of slices of $B_X$ is greater than $\varepsilon.$
	  Since $Y$ has $BSCSP,$ for  $\epsilon > 0,$ there exists slices $S(B_Y,y_i^*,\delta)$ and $ 0 \leq \lambda_i\leq, \sum_{i-1}^{n}  \lambda_i =1$ such that  $dia(\sum_{i=1}^{n} \lambda_i S(B_Y,y_i^*,\delta))<\frac{\varepsilon}{2}.$ 
	  Choose $\tilde{\varepsilon},\delta_0> 0$ such that $\tilde{\varepsilon}+2 \delta_0< \delta,$ and $ 0<\delta_0<\frac{\varepsilon}{8}.$ By Hahn Banach theorem , we can extend $y_i^*$ to a norm preserving extension $x_i^*$ for all $i=1,2,\ldots,n.$
	 Let $P:X\rightarrow X/Y$ be the map such that $P(x)=x+Y.$ 
	 Now, $A_i=P(S(B_X,x_i^*,\tilde{\varepsilon}))$  is a convex subset of $B_{X/Y}$ , since $\|P\|\leqslant 1$ and $A_i$ also contains zero. 
	 Also, by  Proposition \ref{sr}, $\bar{A}_i$=$\ov{(SCS(\bar{A}_i))}$.
	   Thus $\forall i=1,2,\ldots,n,$  there exists $SCS$ point $a_i$  of $\bar{A}_i$  such that $\|a_i\|<\delta_0.$ 
	 Hence $a_i=\sum_{j=1}^{n_i} \gamma_j^i a_j^i \in \sum_{j=1}^{n_i}\gamma_j^i (S(B_{X/Y}, (a_j^i)^*,\eta_j^i)\cap\bar{A}_i$ with $dia(\sum_{j=1}^{n_i}\gamma_j^i (S(B_{X/Y}, (a_j^i)^*,\eta_j^i)\cap\bar{A}_i)<\delta_0$ for all $i=1,2,\ldots,n$ where $\sum_{j=1}^{n_i}\gamma_j^i =1$ with $\gamma_j^i >0$ $\forall j=1,\ldots n_i$, $(a_j^i)^* \in S_{{(X/Y)}^*}$ and $\eta_j^i >0$  for all $i=1,2,\ldots,n.$
	Now,  $S(B_{X/Y}, (a_j^i)^*,\eta_j^i)\cap\bar{A}_i \neq \emptyset$ implies $S(B_{X/Y}, (a_j^i)^*,\eta_j^i)\cap A_i \neq \emptyset.$
	Consequently, $S(B_{X}, P^*(a_j^i)^*,\eta_j^i)\cap S(B_X,x_i^*,\tilde{\varepsilon}) \neq \emptyset.$
	Indeed, let $P(z)\in S(B_{X/Y}, (a_j^i)^*,\eta_j^i)$ for some $z \in S(B_X,x_i^*,\tilde{\varepsilon}).$ hence 
$ P^*(a_j^i)^*(z)=(a_j^i)^*(P(z))>1 - \eta_j^i \geqslant  \sup_{w\in B_{X}} P^*(a_j^i)^* (w) - \eta_j^i.$ 	
	 Now, $ D =\sum_{i=1}^{n} \lambda_i \sum_{j=1}^{n_i}\gamma_j^i ((S(B_{X},P^*(a_j^i)^*,\eta_j^i)\cap S(B_X,x_i^*,\tilde{\varepsilon}))$ is a convex combination of nonempty relatively weakly open subset of $B_X.$
	 By Bourgain's lemma, $D$ contains a convex combination of slices of $B_X$ and since diameter of any convex combination of slices of $B_X$ is greater than $\varepsilon,$ so $dia \sum_{i=1}^{n} \lambda_i \sum_{j=1}^{n_i}\gamma_j^i ((S(B_{X},P^*(a_j^i)^*,\eta_j^i)\cap S(B_X,x_i^*,\tilde{\varepsilon}))>\varepsilon.$ 
	  Then there exists $x_j^i,z_j^i \in (S(B_{X},P^*(a_j^i)^*,\eta_j^i)\cap S(B_X,x_i^*,\tilde{\varepsilon})$ such that $\|\sum_{i=1}^{n} \lambda_i \sum_{j=1}^{n_i}\gamma_j^i x_j^i-\sum_{i=1}^{n} \lambda_i \sum_{j=1}^{n_i}\gamma_j^i   z_j^i \|>\varepsilon.$
	   Note that, $\sum_{j=1}^{n_i}\gamma_j^i x_j^i \in   \sum_{j=1}^{n_i}\gamma_j^i ((S(B_{X},P^*(a_j^i)^*,\eta_j^i)\cap S(B_X,x_i^*,\tilde{\varepsilon}))$
	$\Rightarrow P( \sum_{j=1}^{n_i}\gamma_j^i x_j^i) \in \sum_{j=1}^{n_i}\gamma_j^i (S(B_{X/Y}, (a_j^i)^*,\eta_j^i)\cap A_i ).$
	 Since $diam (\sum_{j=1}^{n_i}\gamma_j^i (S(B_{X/Y}, (a_j^i)^*,\eta_j^i)\cap A_i ) <\delta_0, \forall i=1,2,\ldots,n,$  we have ,\\
	 $\|P( \sum_{j=1}^{n_i}\gamma_j^i x_j^i) - \sum_{j=1}^{n_i}\gamma_j^i a_j^i\|<\delta_0.$
	$\Rightarrow \|P( \sum_{j=1}^{n_i}\gamma_j^i x_j^i)\|<\delta_0+\| \sum_{j=1}^{n_i}\gamma_j^i a_j^i\|$
	$<2\delta_0$ \quad $\forall i=1,2,\ldots,n$\\
	Thus, $d(\sum_{j=1}^{n_i}\gamma_j^i x_j^i,Y)<2\delta_0$
	$\Rightarrow$ for each $i=1,2,\ldots,n$ there exists $v_i\in B_Y$ such that $\|v_i - \sum_{j=1}^{n_i}\gamma_j^i x_j^i\|<2\delta_0.$
	Similarly , for each $i=1,2,\ldots,n$ there exists $w_i\in B_Y$ such that $\|w_i - \sum_{j=1}^{n_i}\gamma_j^i z_j^i\|<2\delta_0.$ 
	Now , $y_i^*(v_i)=y_i^*(\sum_{j=1}^{n_i}\gamma_j^i x_j^i)+y_i^*(v_i-\sum_{j=1}^{n_i}\gamma_j^i x_j^i)>1-\tilde{\varepsilon}-2\delta_0$\\
	Similarly , $ y_i^*(w_i)>1-\tilde{\varepsilon}-2\delta_0.$ Thus $v_i,w_i
	\in S(B_Y,y_i^*,\tilde{\varepsilon}+2\delta_0)$\\ 
	$dia(\sum_{i=1}^{n} \lambda_i  S(B_Y,y_i^*,\tilde{\varepsilon}+2\delta_0) \geqslant \|\sum_{i=1}^{n} \lambda_i v_i - \sum_{i=1}^{n} \lambda_i w_i\|$ 
	$$\geqslant \|\sum_{i=1}^{n} \lambda_i \sum_{j=1}^{n_i}\gamma_j^i x_j^i  - \sum_{i=1}^{n} \lambda_i \sum_{j=1}^{n_i}\gamma_j^i z_j^i  \|-\|\sum_{i=1}^{n} \lambda_i v_i-\sum_{i=1}^{n} \lambda_i \sum_{j=1}^{n_i}\gamma_j^i x_j^i  \|$$  $$-\|\sum_{i=1}^{n} \lambda_i w_i - \sum_{i=1}^{n} \lambda_i \sum_{j=1}^{n_i}\gamma_j^i z_j^i \|$$ 
	$$> \varepsilon - \sum_{i=1}^{n} \lambda_i \|v_i - \sum_{j=1}^{n_i}\gamma_j^i x_j^i   \|-\sum_{i=1}^{n} \lambda_i \|w_i- \sum_{j=1}^{n_i}\gamma_j^i z_j^i   \|$$
	$$>\varepsilon -2\delta_0-2\delta_0$$
	$$=\varepsilon-4\delta_0$$ 
	Since , $\sum_{i=1}^{n} \lambda_i  S(B_Y,y_i^*,\tilde{\varepsilon}+2\delta_0) \subset \sum_{i=1}^{n} \lambda_i  S(B_Y,y_i^*,\delta)$ 
	so $$dia( \sum_{i=1}^{n} \lambda_i  S(B_Y,y_i^*,\delta) \geqslant dia(\sum_{i=1}^{n} \lambda_i  S(B_Y,y_i^*,\tilde{\varepsilon}+2\delta_0))>\varepsilon-4\delta_0 >\varepsilon-\frac{\varepsilon}{2}=\frac{\varepsilon}{2},$$
	a contradiction.
\end{proof}

\begin{proposition}
	Let $Y$ be a closed subspace of $X$ such that $X/Y$ is finite dimensional and $Y$ has $BHP$, then $X$ has $BHP.$
\end{proposition}
\begin{proof}
	Suppose $X$ does not have  $BHP.$ Then there exists $\varepsilon>0$ such that diameter of any relatively weakly open set in $B_X$ is $>\varepsilon.$ 
	Since $Y$ has $BHP$ , so let $W=\{y\in B_Y:|y_i^*(y-y_0)|<\varepsilon_i \quad \forall i=1,2,\ldots,n\}$ where $y_0\in B_Y$ is relatively weakly open set in $B_y$ with diameter $<\frac{\varepsilon}{2}.$ 
	Now for each $i=1,2,\ldots,n$, we can find 
	$\tilde{\varepsilon_i}>0$
	and $\frac{\varepsilon}{4}>\delta_0>0$ such that 
	$\tilde{\varepsilon_i} + \delta_0\|y_i^*\|<\varepsilon_i.$ 
	By Hahn Banach theorem , we can extend $y_i^*$ to a norm preserving extension $x_i^*$ for all $i=1,2,\ldots,n.$ Define \\
	$U=\{x\in B_X: |x_i^*(x-y_0)<\tilde{\varepsilon_i} \quad \forall i=1,2,\ldots,n\}.$
	Clearly $U\neq\emptyset$ and $U$ is relatively weakly open in $B_X.$
	 Let $P:X\rightarrow X/Y$ be the map such that $P(x)=x+Y.$ Then clearly $P$ is onto and linear. 
	Also $P$ is open map by Open Mapping Theorem.
	 Thus $P(U)$ is a norm open set in $X/Y$ and $y_0\in U\cap Y.$ Thus $P(U)$ is a norm open set containing zero. 
	So, there exists $0<\delta<\frac{\delta_0}{2}$ such that $B(0,\delta)\subset P(U).$ 
	Put, $B=P^{-1}(B(0,\delta))\bigcap U \subset B_X.$ 
	Now using the fact that norm norm continuous implies weak weak continuous and $X/Y$ is finite dimensional we can conclude that $B$ is relatively  weakly open in $B_X.$ 
	Then $dia(B)>\varepsilon.$
	 Thus there exists $v_1,v_2\in B$ such that $\|v_1-v_2\|>\varepsilon.$ 
	Now, $\|P(v_1)\| <\delta 
	\Rightarrow d(v_1,Y)<\delta
	\Rightarrow \exists u_1\in Y \ such\ that \|u_1-v_1\|<\delta.$
	 Similarly for $v_2$ there exists $u_2\in Y$ such that $\|u_2-v_2\|<\delta.$ 
	 Without loss of generality we can assume that $u_1,u_2\in B_Y.$ Otherwise we will choose $\frac{u_1}{\|u_1\|}$ and $\frac{u_2}{\|u_2\|}.$ 
	 Now $\forall i=1,2,\ldots,n$ ,
	$|y_i^*(u_1-y_0)|\leqslant |y_i^*(u_1-v_1)|+|y_i^*(v_1-y_0)|\leqslant \|y_i^*\| 2 \delta + \tilde{\varepsilon_i} <\| y_i^*\| \delta_0 + \tilde{\varepsilon_i}<\varepsilon_i$ \\
	Thus, $u_1\in W.$ 
	Similarly $u_2\in W$ and \\$\|u_1-u_2\|\geqslant \|v_1-v_2\|+\|u_1-v_1\|+\|u_2-v_2\|$  $>\varepsilon-4\delta>
	\varepsilon-2\delta_0>\varepsilon-\frac{\varepsilon}{2}=\frac{\varepsilon}{2}.$ \\ 
	Thus , $dia(W)>\frac{\varepsilon}{2}$ , a contradiction.
\end{proof}


\begin{Acknowledgement}
The second  author's research is funded by the National Board for Higher Mathematics (NBHM), Department of Atomic Energy (DAE), Government of India, Ref No: 0203/11/2019-R$\&$D-II/9249.
\end{Acknowledgement}

\end{document}